\documentclass[final,3p,times]{elsarticle}

\usepackage{amssymb}
\usepackage{amsthm}
\usepackage{amsmath}

\journal{Applied and Computational Harmonic Analysis}

\newcommand{\bbC}{\mathbb{C}}
\newcommand{\bbR}{\mathbb{R}}
\newcommand{\bbZ}{\mathbb{Z}}
\newcommand{\bbH}{\mathbb{H}}
\newcommand{\bbS}{\mathbb{S}}

\newcommand{\cI}{\mathcal{I}}
\newcommand{\cJ}{\mathcal{J}}
\newcommand{\rme}{\mathrm{e}}
\newcommand{\rmi}{\mathrm{i}}
\newcommand{\rmE}{\mathrm{E}}
\newcommand{\rmI}{\mathrm{I}}
\newcommand{\rmT}{\mathrm{T}}
\newcommand{\Tr}{\mathrm{Tr}}

\newcommand{\real}{\mathrm{Re}}
\newcommand{\HS}{\mathrm{HS}}
\newcommand{\FP}{\mathrm{FP}}

\newcommand{\abs}[1]{|{#1}|}
\newcommand{\bigabs}[1]{\bigl|{#1}\bigr|}

\newcommand{\biggabs}[1]{\biggl|{#1}\biggr|}

\newcommand{\paren}[1]{({#1})}
\newcommand{\bigparen}[1]{\bigl({#1}\bigr)}
\newcommand{\Bigparen}[1]{\Bigl({#1}\Bigr)}
\newcommand{\biggparen}[1]{\biggl({#1}\biggr)}

\newcommand{\bigbracket}[1]{\bigl[{#1}\bigr]}
\newcommand{\Bigbracket}[1]{\Bigl[{#1}\Bigr]}
\newcommand{\biggbracket}[1]{\biggl[{#1}\biggr]}

\newcommand{\set}[1]{\{{#1}\}}
\newcommand{\bigset}[1]{\bigl\{{#1}\bigr\}}

\newcommand{\biggset}[1]{\biggl\{{#1}\biggr\}}

\newcommand{\norm}[1]{\|{#1}\|}
\newcommand{\bignorm}[1]{\bigl\|{#1}\bigr\|}

\newcommand{\biggnorm}[1]{\biggl\|{#1}\biggr\|}

\newcommand{\ip}[2]{\langle{#1},{#2}\rangle}

\newcommand{\Bigip}[2]{\Bigl\langle{#1},{#2}\Bigr\rangle}
\newcommand{\biggip}[2]{\biggl\langle{#1},{#2}\biggr\rangle}

\newlength{\wsla}
\newlength{\wslb}
\newlength{\wslc}
\newlength{\wsld}

\newtheorem{thm}{Theorem}

\newtheorem{prop}[thm]{Proposition}
\newtheorem{cor}[thm]{Corollary}
\newproof{pf}{Proof}

\theoremstyle{definition}
\newtheorem{exmp}[thm]{Example}
\newtheorem{defn}[thm]{Definition}

\begin{document}
\begin{frontmatter}
\title{Auto-tuning unit norm frames}

\author[MO]{Peter G.~Casazza}
\author[AFIT]{Matthew Fickus}
\ead{Matthew.Fickus@afit.edu}
\author[Princeton]{Dustin G.~Mixon}

\address[MO]{Department of Mathematics, University of Missouri, Columbia, Missouri 65211, USA}
\address[AFIT]{Department of Mathematics and Statistics, Air Force Institute of Technology, Wright-Patterson Air Force Base, Ohio 45433, USA}
\address[Princeton]{Program in Applied and Computational Mathematics , Princeton University, Princeton, New Jersey 08544, USA}

\begin{abstract}
Finite unit norm tight frames provide Parseval-like decompositions of vectors in terms of redundant components of equal weight.  They are known to be exceptionally robust against additive noise and erasures, and as such, have great potential as encoding schemes.  Unfortunately, up to this point, these frames have proven notoriously difficult to construct.  Indeed, though the set of all unit norm tight frames, modulo rotations, is known to contain manifolds of nontrivial dimension, we have but a small finite number of known constructions of such frames.  In this paper, we present a new iterative algorithm---gradient descent of the frame potential---for increasing the degree of tightness of any finite unit norm frame.  The algorithm itself is trivial to implement, and it preserves certain group structures present in the initial frame.  In the special case where the number of frame elements is relatively prime to the dimension of the underlying space, we show that this algorithm converges to a unit norm tight frame at a linear rate, provided the initial unit norm frame is already sufficiently close to being tight.  By slightly modifying this approach, we get a similar, but weaker, result in the non-relatively-prime case, providing an explicit answer to the Paulsen problem: ``How close is a frame which is almost tight and almost unit norm to some unit norm tight frame?''
\end{abstract}

\end{frontmatter}


\section{Introduction}

\textit{Frames} provide numerically stable methods for finding overcomplete decompositions of vectors, and are ubiquitous in signal processing applications~\cite{KovacevicC:07a,KovacevicC:07b}.  As explained below, \textit{tight frames} and \textit{unit norm} frames are particularly useful.  However, it is difficult to construct frames which possess both of these properties simultaneously, called \textit{unit norm tight frames} (UNTFs).  In this paper, we present a new method for overcoming this difficulty, namely an iterative procedure which, when applied to a given finite unit norm frame, asymptotically produces a UNTF.  To be precise, under the additional assumptions that the number of frame vectors is relatively prime to the dimension of the underlying space and that our initial unit norm frame is sufficiently close to being tight, we are able to show that our method, namely a gradient descent of the \textit{frame potential}, converges to a UNTF at a linear rate.  That is, from a tightness perspective, our algorithm takes a good unit norm frame and makes it perfect.  As such, it can be viewed as a frame-theoretic analog of \textit{Auto-Tune}\texttrademark, the software commonly used in the music industry to perfect the pitch of lesser vocalists.  Moreover, in the non-relatively-prime case, we can slightly modify our argument to yield an explicit answer to the \textit{Paulsen problem}~\cite{BodmannC:10}: 
\begin{equation*}
\text{\emph{``How close is a frame which is almost tight and almost unit norm to some UNTF?''}}
\end{equation*}

To make these notions precise, consider the \textit{synthesis operator} of a sequence of vectors $F=\set{f_n}_{n=1}^N$ in a real or complex $M$-dimensional Hilbert space $\bbH_M$, namely $F:\bbC^N\rightarrow\bbH_M$, $\smash{Fg:=\sum_{n=1}^N g(n)f_n}$.  That is, viewing $\bbH_M$ as $\bbR^M$ or $\bbC^M$, $F$ is the $M\times N$ matrix whose columns are the $f_n$'s.  Note that here and throughout, we make no notational distinction between the vectors themselves and the synthesis operator they induce.  The vectors $F$ are said to be a \textit{frame} for $\bbH_M$ if there exists \textit{frame bounds} $0<A\leq B<\infty$ such that $A\norm{f}^2\leq\norm{F^*f}^2\leq B\norm{f}^2$ for all $f\in\bbH_M$.  In this finite-dimensional setting, having $F$ be a frame is equivalent to having the $f_n$'s span $\bbH_M$, necessitating $M\leq N$, with the optimal frame bounds $A$ and $B$ corresponding to the least and greatest eigenvalues of $FF^*$.  In particular, $F$ is a \textit{tight frame} when $A=B$, that is, when $FF^*=A\rmI$.  Tight frames are useful in applications, as they provide Parseval-like decompositions
\begin{equation}
\label{equation.definition of tight frame}
f
=\tfrac1AFF^*f
=\tfrac1A\sum_{n=1}^N\ip{f}{f_n}f_n,\quad \forall f\in\bbH_M,
\end{equation}
despite the fact that the $f_n$'s are not required to be independent.  Indeed, the tightness condition $FF^*=A\rmI$ does not require the columns of $F$, that is, the $f_n$'s, to be orthogonal, but rather, it requires the rows of $F$ to be orthogonal and have equal norm $\sqrt{A}$.  Meanwhile, $F$ is a \textit{unit norm} frame when $\norm{f_n}=1$ for all $n=1,\dotsc,N$.  When a frame is both unit norm and tight---a UNTF---it breaks vectors into possibly redundant components of equal weight~\eqref{equation.definition of tight frame}, with the tight frame constant $A$ being the redundancy $\frac NM$.  UNTFs are known to be exceptionally robust against additive noise and erasures~\cite{CasazzaK:03,GoyalKK:01,GoyalVT:98,HolmesP:04}.  Unfortunately, UNTFs are also notoriously difficult to construct: we want $M\times N$ matrices $F$ that have unit norm columns and orthogonal rows of equal squared-norm $\frac NM$.  To be clear, UNTFs are known to exist for any $M\leq N$: one may either invoke the classical theory of \textit{majorization} for matrices, or more simply, consider the \textit{harmonic frame} obtained by truncating an $N\times N$ discrete Fourier transform (DFT) matrix~\cite{GoyalKK:01}.  Another technique is to build an operator with a flat spectrum using weighted DFT blocks; this \textit{spectral tetris} method yields extremely sparse UNTFs~\cite{CasazzaFMWZ:10}.  However, these techniques only produce certain examples of UNTFs, while the set of all UNTFs, modulo rotations, contains nontrivial manifolds whenever $N>M+1$~\cite{DykemaS:06}.  That is, these methods produce but a few samples from the continuum.

In this paper, we provide a new method for starting with a given frame and producing a nearby UNTF from it.  Such techniques are very useful in real-world problems, as they allow one to take a given transform, carefully crafted to have certain application-specific properties without being tight and/or unit norm, and to correct, or \textit{tune}, its algebraic properties while changing the transform itself as little as possible.  In terms of mathematics, these techniques are important because they help in solving the Paulsen problem.  To be precise, a compactness argument of D.~Hadwin~\cite{BodmannC:10} shows that indeed, if a frame is sufficiently close to being both tight and unit norm, then it is, in fact, close to a UNTF.  Current work on this problem therefore focuses on \textit{how} close these UNTFs are, as well as developing practical schemes to obtain them.  Unfortunately, finitely-iterative techniques using Givens rotations~\cite{CasazzaL:06,HolmesP:04} have, to this point, produced UNTFs that are not necessarily close to the originals.

More recent approaches to solving the Paulsen problem, namely that of~\cite{BodmannC:10} and the present method, rely upon the fact that given any frame $F$, it is straightforward to produce a unit norm frame from it: simply replace each $f_n$ with $\tfrac{f_n}{\norm{f_n}}$.  Moreover, one can also convert any frame into a tight frame, provided one has the computational power to take the inverse square root of the frame operator: consider $(FF^*)^{-\frac12}F$.  However, combining these two operations---dividing by the root of the frame operator and then normalizing the resulting vectors, or vice versa---does not yield UNTFs, as these two operations do not commute.  Nevertheless, by using one of these two techniques, one may assume without loss of generality~\cite{BodmannC:10} that either the initial frame is exactly tight and nearly unit norm or, alternatively, that the initial frame is exactly unit norm and nearly tight.  The former approach is that taken by~\cite{BodmannC:10}: starting with a tight frame that is not unit norm, they solve a differential equation that minimizes \textit{frame energy} while preserving tightness, flowing towards a UNTF; this led to the first genuine solution to the Paulsen problem in the special case where $M$ and $N$ are relatively prime.  The latter approach is the one we pursue here.

In particular, starting with a frame that is already unit norm, we try to produce a UNTF from it.  Preliminary results to this end were reported in the conference proceedings paper~\cite{CasazzaF:09b}.  We accomplish this task by descending against the gradient of the \textit{frame potential}, namely the square of the Hilbert-Schmidt norm of the Gram matrix $F^*F$, regarded as a function over $N$ copies of the unit sphere $\mathbb{S}_M:=\{f\in\bbH_M:\norm{f}=1\}$:
\begin{equation*}
\FP:\mathbb{S}_M^N\rightarrow\mathbb{R}, \quad \FP(F)=\norm{F^*F}_\HS^2=\sum_{n=1}^N\sum_{n'=1}^N\abs{\ip{f_n}{f_{n'}}}^2.
\end{equation*}
Introduced in~\cite{BenedettoF:03}, the frame potential is the total potential energy contained within a given collection of points on the sphere under the action of a \textit{frame force} which encourages orthogonality.  As discussed in the next section, one can show that $\smash{\FP(F)=\frac{N^2}M+\norm{FF^*-\tfrac NM\rmI}_\HS^2}$ for any $F\in\mathbb{S}_M^N$.  That is, the frame potential is bounded below by $\smash{\tfrac{N^2}M}$, with equality if and only if $F$ is a UNTF.  The main result of~\cite{BenedettoF:03} gives that even \textit{local} minimizers of $\FP$ are UNTFs.  As such, even if no explicit constructions of such frames were known, they must exist: $\FP$ is a continuous function over the compact set $\mathbb{S}_M^N$, and as such, possesses a global minimizer, which is necessarily a local minimizer, which is necessarily a UNTF.  This existence argument has been generalized to numerous other settings~\cite{CasazzaF:09,CasazzaFKLT:06,FickusJKO:05,JohnsonO:08,Massey:09,MasseyR:10,MasseyRS:10}.  Moreover, this fact implies that every local minimizer of $\FP$ is necessarily a global minimizer, which is a nice property to have when performing gradient descent; even here, this task is nontrivial however, as there are nonoptimal arrangements at which the first derivative of the frame potential vanishes~\cite{BenedettoF:03}.

The novelty and significance of our work is best gauged by contrasting it with the current state-of-the-art of the Paulsen problem: the technique of~\cite{BodmannC:10}.  Both approaches give valid solutions to the Paulsen problem and have certain applications for which they are preferable to the other.  Instead of assuming our frame is already tight and seeking to become increasingly unit norm~\cite{BodmannC:10}, we assume we are already unit norm and seek tightness.  Rather than needing to solve a differential equation~\cite{BodmannC:10}, we have an iterative, gradient-descent-based algorithm; our approach only becomes a differential equation when the step size is forced arbitrarily small.  While the relative primeness of $M$ and $N$ is an important consideration in both methods, the technique of~\cite{BodmannC:10} is only guaranteed to converge in this case, while our convergence argument generalizes to the non-relatively-prime case, albeit in a weaker form.  Also, as shown below, our method preserves the group structure of certain UNTF constructions, such as Gabor frames and filter banks, whereas~\cite{BodmannC:10} does not.

In the next section, we introduce the fundamental concepts needed to compute the gradient of the frame potential (Theorem~\ref{theorem.gradient of the frame potential}) and study its group invariance properties (Proposition~\ref{proposition.group invariance}).  In Section 3, we find sufficient conditions that guarantee that gradient descent of the frame potential converges to a UNTF at a linear rate (Theorem~\ref{theorem.linear convergence}).  In the fourth and final section, we show that these sufficient conditions are indeed met provided $M$ and $N$ are relatively prime and the initial frame is already sufficient tight, yielding an answer to the Paulsen problem in this case (Corollary~\ref{corollary.relatively prime}).  We further discuss how these arguments generalize to the non-relatively-prime case (Theorem~\ref{theorem.not relatively prime}).

\section{The gradient of the frame potential}

In this section, we lay the groundwork for our approach to modify a given unit norm frame so as to decrease its distance from tightness.  As such, our first priority is to formally define this distance.  Let $\set{\lambda_m}_{m=1}^{M}$ be the eigenvalues of the frame operator $FF^*$ of some unit norm sequence  $F=\set{f_n}_{n=1}^{N}$.  Note that since
\begin{equation*}
\sum_{m=1}^{M}\lambda_m=\Tr(FF^*)=\Tr(F^*F)=\sum_{n=1}^{N}\norm{f_n}^2=N,
\end{equation*}
the average value of these eigenvalues is $\frac NM$.  Moreover, $F$ is a UNTF if and only if $FF^*=\frac NM\rmI$, that is, if and only if all the $\lambda_m$'s are equal to $\frac MN$.  As such, in the past, the \textit{distance from tightness} of a unit norm frame $F$ has usually been defined as $\smash{\max_m\abs{\lambda_m-\frac NM}}$.  However, as there is no closed-form expression for eigenvalues exist, we propose an alternative measure of tightness, namely the $2$-norm of the values $\set{\lambda_m-\frac NM}_{m=1}^M$:
\begin{equation}
\label{equation.distance from tightness}
\sum_{m=1}^{M}\bigparen{\lambda_m-\tfrac NM}^2
=\bignorm{FF^*-\tfrac NM\rmI}_\HS^2
=\Tr\bigbracket{(FF^*)^2}-2\tfrac NM\Tr(FF^*)+\tfrac{N^2}{M^2}\Tr(\rmI)
=\FP(F)-\tfrac{N^2}M.
\end{equation}
In particular, we see that $\FP(F)\geq\tfrac{N^2}M$, with equality if and only if $F$ is a UNTF.  It therefore makes sense to define our notion of the \textit{distance from tightness} of $F$ to be the easily computable quantity $\smash{\norm{FF^*-\tfrac NM\rmI}_\HS=\bigparen{\FP(F)-\tfrac{N^2}M}^{\frac12}}$.  Written in this language, the version of the Paulsen problem on which we focus is the following:  
\begin{quote}
\textit{Given positive integers $M$ and $N$, find possibly $(M,N)$-dependent constants $\delta$, $C$ and $\alpha$ such that given any unit norm sequence $F$ such that $\norm{FF^*-\tfrac NM\rmI}_\HS\leq\delta$, there necessarily exists a UNTF $\tilde{F}$ such that}
\begin{equation}
\label{equation.our version of Paulsen}
\norm{\tilde{F}-F}_\HS\leq C\bignorm{FF^*-\tfrac NM\rmI}_\HS^\alpha.
\end{equation}
\end{quote}

One way to get a ballpark estimate on what these parameters $\delta$, $C$ and $\alpha$ should be, under the best possible circumstances, is to solve a weaker problem: given a unit norm frame $F$, find $\tilde{F}$ such that $\tilde{F}\tilde{F}^*=\frac NM\rmI$ and such that $\smash{\norm{\tilde{F}-F}_\HS}$ is minimized; here, we do not require that $\tilde{F}$ be unit norm.  Similar problems have been extensively studied in the past---see \cite{BodmannC:10} for references.  In brief, we have that for any such $\tilde{F}$ and $F$, $\smash{\norm{\tilde{F}-F}_\HS^2=2N-2\real\Tr(\tilde{F}^*F)}$.  Taking the singular value decomposition $F=U\Sigma V$ and letting $\tilde{\Sigma}=U^*\tilde{F}V^*$ so that $\tilde{F}=U\tilde{\Sigma}V$, we are therefore seeking to maximize $\real\Tr(\tilde{F}^*F)=\real\Tr(\tilde{\Sigma}^*\Sigma)$ subject to the restriction that $\tilde{\Sigma}\tilde{\Sigma}^*=\frac NM\rmI$.  As $\Sigma$ is ``diagonal," this maximum is achieved by letting $\tilde{\Sigma}$ also be ``diagonal" with entries $\smash{(\tfrac NM)^{\frac12}}$, implying
\begin{equation*}
\norm{\tilde{F}-F}_{\HS}^2
=2N-2\real\Tr(\tilde{\Sigma}^*\Sigma)
\geq 2N-2(\tfrac NM)^{\frac12}\sum_{m=1}^{M}\lambda_m^{\frac12}
=\sum_{m=1}^{M}\Bigbracket{\lambda_m^{\frac12}-(\tfrac NM)^{\frac12}}^2.
\end{equation*}
Multiplying the terms in these summands by their conjugates $\smash{\lambda_m^{\frac12}+(\tfrac NM)^{\frac12}}$ then yields
\begin{equation*}
\norm{\tilde{F}-F}_{\HS}^2
\geq\sum_{m=1}^{M}\frac{\Bigparen{\lambda_m-\tfrac NM}^2}{\Bigbracket{\lambda_m^{\frac12}+(\tfrac NM)^{\frac12}}^2}
\geq\tfrac{M}{N}\sum_{m=1}^{M}(\lambda_m-\tfrac NM)^2
=\tfrac{M}{N}\bignorm{FF^*-\tfrac NM\rmI}_\HS^2.
\end{equation*}
To summarize, the UNTF $\tilde{F}$ which is closest to $F$ necessarily satisfies $\smash{\norm{\tilde{F}-F}_{\HS}\geq(\tfrac MN)^{\frac12}}\norm{FF^*-\tfrac NM\rmI}_\HS$.  As such, in our version of the Paulsen problem~\eqref{equation.our version of Paulsen}, the best $\alpha$ we should expect is $\alpha=1$.  Indeed, in the case where $M$ and $N$ are relatively prime, we show that $\alpha=1$ is achievable, provided $\delta$ and $C$ are suitably chosen.  Meanwhile, when $M$ and $N$ have a common divisor, a simple example, given in Section 4, shows that the best one can expect is $\alpha=\frac12$.  As we shall see, the key issue with the non-relatively-prime case is that there exist UNTFs which can be partitioned into mutually orthogonal subcollections; at such frames, the geometric structure of the set of surrounding UNTFs is extremely complicated~\cite{DykemaS:06}.

\subsection{The gradient of the frame potential}

Now that we have formally defined the distance from tightness of a unit norm frame $F$ to be $\norm{FF^*-\frac NM\rmI}_\HS$, and having further posed the problem we are trying to solve with~\eqref{equation.our version of Paulsen}, we turn to our specific approach:  a gradient descent of the squared distance from tightness, which, since $\smash{\norm{FF^*-\frac NM\rmI}_\HS^2=\FP(F)-\tfrac{N^2}M}$, reduces to a gradient descent of the frame potential.  Here, as the domain of optimization $\mathbb{S}_M^N$ is a product of spheres as opposed to the entire space $\bbH_M^N$, this version of gradient descent differs from the one most commonly used.  In particular, given $F=\set{f_n}_{n=1}^{N}$ in $\mathbb{S}_M^N$ and $G=\set{g_n}_{n=1}^{N}$ in $\oplus_{n=1}^N f_n^\perp:=\bigset{\set{g_n}_{n=1}^{N}\in\bbH_M^N: \ip{f_n}{g_n}=0,\ \forall n}$, we use Lemma~2 of~\cite{CasazzaF:09} along with Taylor's theorem to estimate the change in frame potential as each $f_n$ is pushed along a great circle with tangent velocity $g_n$:

\begin{prop}
\label{proposition.Taylor's theorem}
For any $F=\set{f_n}_{n=1}^{N}\in \mathbb{S}_M^N$ and $G=\set{g_n}_{n=1}^{N}\in\oplus_{n=1}^N f_n^\perp$, let $\smash{f_n(t):=\cos(\norm{g_n}t)f_n-\sin(\norm{g_n}t)\frac{g_n}{\norm{g_n}}}$ whenever $g_n\neq0$, and let $f_n(t):=f_n$ otherwise. Then, $F(t)=\set{f_n(t)}_{n=1}^{N}\in \mathbb{S}_M^N$ for any $t\in\mathbb{R}$ and satsifies
\begin{align}
\label{equation.Taylor's theorem 1}
\norm{F(t)-F}_\HS^2
&\leq t^2\sum_{n=1}^{N}\norm{g_n}^2,\\
\label{equation.Taylor's theorem 2}
\FP(F(t))
&\leq\FP(F)-4t\real\sum_{n=1}^{N}\ip{FF^*f_n}{g_n}+8Nt^2\sum_{n=1}^{N}\norm{g_n}^2.
\end{align}
\end{prop}

\begin{proof}
It is straightforward to show that $\norm{f_n(t)}=1$ for all $n=1,\dotsc,N$ and all $t\in\mathbb{R}$. To show~\eqref{equation.Taylor's theorem 1}, note that for any $n$ such that $g_n\neq0$, we have
\begin{equation}
\label{equation.proof of Taylor's theorem 1}
\norm{f_n(t)-f_n}^2
=\bigparen{\cos(\norm{g_n}t)-1}^2+\sin^2(\norm{g_n}t)
=4\sin^2(\norm{g_n}t/2)
\leq\norm{g_n}^2t^2.
\end{equation}
As~\eqref{equation.proof of Taylor's theorem 1} also immediately holds for any $n$ such that $g_n=0$, we may sum~\eqref{equation.proof of Taylor's theorem 1} over all $n$ to conclude~\eqref{equation.Taylor's theorem 1}.  To prove~\eqref{equation.Taylor's theorem 2}, we apply Taylor's theorem to $\varphi(t)=\FP(F(t))$ at $t=0$:
\begin{equation}
\label{equation.proof of Taylor's theorem 2}
\varphi(t)\leq\varphi(0)+t\dot{\varphi}(0)+\tfrac{1}2t^2\max_{s\in\mathbb{R}}\abs{\ddot{\varphi}(s)}.
\end{equation}
To compute the terms in~\eqref{equation.proof of Taylor's theorem 2}, note that $\dot{f}_n(t)=-\norm{g_n}\sin(\norm{g_n}t)f_n-\cos(\norm{g_n}t)g_n$ for any $n$ such that $g_n\neq0$, a fact that also holds trivially when $g_n=0$, since $f_n(t)$ is constant.  In particular, $\dot{f}_n(0)=-g_n$ for all $n=1,\dotsc,N$.  The expression for $\dot{\varphi}(t)$ given in Lemma~2 of \cite{CasazzaF:09} then gives
\begin{equation}
\label{equation.proof of Taylor's theorem 3}
\dot{\varphi}(0)
=4\real\Tr\bigparen{\dot{F}^*(0)F(0)F^*(0)F(0)}
=4\real\Tr\bigparen{-G^*FF^*F}
=-4\real\sum_{n=1}^{N}\ip{G^*FF^*Fe_n}{e_n}
=-4\real\sum_{n=1}^{N}\ip{FF^*f_n}{g_n},
\end{equation}
where $\set{e_n}_{n=1}^N$ is the standard basis of $\bbH_N$.  Next, as $\ddot{f}_n(t)=-\norm{g_n}^2f_n(t)$ for any $n$, we further have
\begin{equation}
\label{equation.proof of Taylor's theorem 4}
\Tr(\ddot{F}^*(t)F(t)F^*(t)F(t))
=\sum_{n=1}^{N}\ip{\ddot{F}^*(t)F(t)F^*(t)F(t)e_n}{e_n}
=\sum_{n=1}^{N}\ip{F^*(t)f_n(t)}{F^*(t)\ddot{f_n}(t)}
=-\sum_{n=1}^{N}\norm{g_n}^2\norm{F^*(t)f_n(t)}^2.
\end{equation}
Substituting~\eqref{equation.proof of Taylor's theorem 4} into the expression for $\ddot{\varphi}(t)$ given in Lemma~2 of \cite{CasazzaF:09} yields
\begin{equation}
\label{equation.proof of Taylor's theorem 5}
\ddot{\varphi}(t)=-4\sum_{n=1}^{N}\norm{g_n}^2\norm{F^*(t)f_n(t)}^2+4\norm{\dot{F}^*(t)F(t)}_\HS^2+2\norm{\dot{F}(t)F^*(t)+F(t)\dot{F}^*(t)}_\HS^2.
\end{equation}
To bound~\eqref{equation.proof of Taylor's theorem 5}, note that $\norm{F(t)}_\HS^2=\sum_{n=1}^{N}\norm{f_n(t)}^2=N$ and $\norm{\dot{F}(t)}_\HS^2=\sum_{n=1}^{N}\norm{\dot{f}_n(t)}^2=\sum_{n=1}^{N}\norm{g_n}^2$, and thus
\begin{align}
\nonumber
\abs{\ddot{\varphi}(t)}
&\leq4\sum_{n=1}^{N}\norm{g_n}^2\norm{F^*(t)f_n(t)}^2+4\norm{\dot{F}^*(t)F(t)}_\HS^2+2\norm{\dot{F}(t)F^*(t)+F(t)\dot{F}^*(t)}_\HS^2\\
\nonumber
&\leq4\sum_{n=1}^{N}\norm{g_n}^2\norm{F(t)}_2^2\norm{f_n(t)}^2+4\norm{\dot{F}^*(t)F(t)}_\HS^2+2\Bigparen{\norm{\dot{F}(t)F^*(t)}_\HS+\norm{F(t)\dot{F}^*(t)}_\HS}^2\\
\nonumber
&\leq4\sum_{n=1}^{N}\norm{g_n}^2\norm{F(t)}_\HS^2+12\norm{\dot{F}(t)}_\HS^2\norm{F(t)}_\HS^2\\
\label{equation.proof of Taylor's theorem 7}
&=16N\sum_{n=1}^{N}\norm{g_n}^2.
\end{align}
Substituting~\eqref{equation.proof of Taylor's theorem 3} and~\eqref{equation.proof of Taylor's theorem 7} into~\eqref{equation.proof of Taylor's theorem 2} yields~\eqref{equation.Taylor's theorem 2}.
\end{proof}

Considering the Taylor expansion of $\FP(F(t))$ given in~\eqref{equation.Taylor's theorem 1}, one might expect the \textit{gradient} of $\FP$ over $\mathbb{S}_M^N$, namely the choice of vectors $\set{g_n}_{n=1}^{N}$, modulo positive scalar multiples, which maximizes the linear term $\smash{\real\sum_{n=1}^{N}\ip{FF^*f_n}{g_n}}$,  to be given by $g_n=FF^*f_n$ for all $n=1,\dotsc,N$. Indeed, one may show that this would be the correct gradient if we regarded the frame potential as a functional over the entire space $\bbH_M^N$. However, since we are optimizing over $\mathbb{S}_M^N$, we require that $\set{g_n}_{n=1}^{N}\in\oplus_{n=1}^N f_n^\perp$.  Therefore, we instead take $\set{g_n}_{n=1}^N$ to be the projection of $\set{FF^* f_n}_{n=1}^{N}$ onto $\oplus_{n=1}^N f_n^\perp$.  In the next result, we formally verify that such a choice is optimal.

\begin{thm}
\label{theorem.gradient of the frame potential}
Pick $F=\set{f_n}_{n=1}^{N}\in\mathbb{S}_M^N$,
and let $P_n$ denote the orthogonal projection from $\mathbb{H}_M$ onto the orthogonal complement of $f_n$.
Then, the minimizer of the bound in~\eqref{equation.Taylor's theorem 2} over all $t\in\mathbb{R}$ and $\set{g_n}_{n=1}^{N}\in\oplus_{n=1}^N f_n^\perp$ is given by $t=\tfrac{1}{4N}$ and
\begin{equation}
\label{equation.optimal search directions}
g_n=P_nFF^*f_n=FF^*f_n-\ip{FF^*f_n}{f_n}f_n,\quad n=1,\dotsc,N.
\end{equation}
Moreover, for any $t\in\mathbb{R}$, this choice for $\set{g_n}_{n=1}^{N}$ gives
\begin{align}
\label{equation.gradient descent 1}
\norm{F(t)-F}_\HS^2
&\leq t^2 \sum_{n=1}^{N}\norm{P_nFF^*f_n}^2,\\
\label{equation.gradient descent 2}
\FP(F(t))
&\leq \FP(F)-4t\bigparen{1-2Nt}\sum_{n=1}^{N}\norm{P_nFF^*f_n}^2.
\end{align}
\end{thm}

\begin{proof}
We seek to minimize
\begin{equation}
\label{equation.proof of gradient descent 1}
-4t\real\sum_{n=1}^{N}\ip{FF^*f_n}{g_n}+8Nt^2\sum_{n=1}^{N}\norm{g_n}^2=\tfrac2N\sum_{n=1}^{N}\real\ip{-FF^*f_n+2Ntg_n}{2Ntg_n}
\end{equation}
over all $\set{g_n}_{n=1}^{N}\in\mathbb{S}_M^N$ and all $t\in\mathbb{R}$. We note immediately from~\eqref{equation.proof of gradient descent 1} that the optimal $\set{g_n}_{n=1}^{N}$ and $t$ are not unique, though we now show that their product is.  Indeed, we have $P_ng_n=g_n$, and therefore
\begin{align*}
\real\ip{-FF^*f_n+2Ntg_n}{2Ntg_n}
&=\real\ip{-FF^*f_n+2Ntg_n}{2NtP_ng_n}\\
&=\real\ip{-P_nFF^*f_n+2Ntg_n}{2Ntg_n}\\
&=\tfrac14\bigparen{\norm{-P_nFF^*f_n+4Ntg_n}^2-\norm{-P_nFF^*f_n}^2}\\
&\geq-\tfrac14\norm{P_nFF^*f_n}^2,
\end{align*}
with equality if and only if $-P_nFF^*f_n+4Ntg_n=0$. Thus, to minimize~\eqref{equation.proof of gradient descent 1}, and consequently to minimize the upper bound in~\eqref{equation.Taylor's theorem 2}, we may take $t=\tfrac{1}{4N}$ and $g_n=P_nFF^*f_n$, as claimed.
Moreover, substituting these choices of $g_n$'s into~\eqref{equation.Taylor's theorem 1} and~\eqref{equation.Taylor's theorem 2} yields~\eqref{equation.gradient descent 1} and~\eqref{equation.Taylor's theorem 2}, respectively.
\end{proof}

Note that for any $t\in(0,\tfrac{1}{2N})$, Theorem~\ref{theorem.gradient of the frame potential} prescribes a direction and step size to travel from a given $F\in\mathbb{S}_M^N$ which guarantees a predictable decrease in frame potential.  Throughout the remainder of this paper, we fix any such $t$ and repeatedly apply Theorem~\ref{theorem.gradient of the frame potential} to produce a sequence of iterations which, in many cases, is guaranteed to converge to a UNTF.  One may also consider what happens to this sequence of iterations as $t$ is taken ever smaller; as $t\rightarrow0$, we expect to approach a solution to the system of nonlinear ordinary differential equations:
\begin{equation*}
\smash{\dot{f}_n(s)=-\,\Bigparen{\,F(s)F^*(s)f_n(s)-\Bigip{F(s)F^*(s)f_n(s)}{f_n(s)}f_n(s)f_n(s)\,},\quad \forall n=1,\dotsc,N,}
\end{equation*}  
a matter we leave for future research.

\subsection{The preservation of group structure}

Many popular examples of unit norm frames, such as oversampled filter banks and Gabor frames, have a group structure.  In particular, such frames are the \textit{orbit} $\set{U_i f_j}_{i\in\cI, j\in\cJ}$ of a collection of unit vectors $\set{f_j}_{j\in\cJ}$ under the action of a collection of unitary operators $\set{U_i}_{i\in\cI}$.  While such frames inherently consist of unit norm vectors, it can be difficult to ensure their tightness~\cite{ChebiraFM:10,FickusJKO:05}.  As such, it would be valuable to have a technique which increases the tightness of such frames without sacrificing their group structure.  The next result shows that the technique of Theorem \ref{theorem.gradient of the frame potential} does precisely this, provided the unitary operators are known to commute with the frame operator.

\begin{prop}
\label{proposition.group invariance}
Let the orbit $F=\set{f_{i,j}}_{i\in\cI, j\in\cJ}=\set{U_i f_j}_{i\in\cI, j\in\cJ}$ of unit vectors have the property that every unitary matrix $U_i$ commutes with its frame operator $FF^*$.  Then, pushing these vectors along the tangent directions $\set{g_{i,j}}_{i\in\cI, j\in\cJ}$ given in~\eqref{equation.optimal search directions} produces new collections of vectors which possess this same group structure: $F(t)=\set{U_i f_j(t)}_{i\in\cI, j\in\cJ}$.
\end{prop}

\begin{proof}
We have $f_{i,j}(t)=\cos(\norm{g_{i,j}}t)f_{i,j}-\sin(\norm{g_{i,j}}t)\frac{g_{i,j}}{\norm{g_{i,j}}}$ where $g_{i,j}:=P_{i,j}FF^* f_{i,j}$.  That is,
\begin{equation*}
g_{i,j}
=FF^*U_i f_j-\Bigip{FF^*U_i f_j}{U_i f_j}U_i f_j
=U_iFF^*f_j-\Bigip{U_iFF^*f_j}{U_i f_j}U_i f_j
=U_i\Bigparen{FF^*f_j-\Bigip{FF^*f_j}{f_j}f_j}
=U_i g_j,
\end{equation*}
where $g_j:=FF^*f_j-\ip{FF^*f_j}{f_j}f_j$.  We thus have that $f_{i,j}(t)=U_j f_i(t)$, as claimed:
\begin{equation*}
f_{i,j}(t)
=\cos(\norm{U_i g_j}t)U_i f_j-\sin(\norm{U_i g_j}t)\tfrac{U_i g_j}{\norm{U_i g_j}}
=U_i\Bigparen{\cos(\norm{g_j}t)f_j-\sin(\norm{g_j}t)\tfrac{g_j}{\norm{g_j}}}
=U_i f_j(t).\qedhere
\end{equation*}
\end{proof}

For example, consider the space of discrete $M$-periodic signals $\ell(\bbZ_M)=\set{f:\bbZ\rightarrow\bbC: f(m+M)=f(m),\ \forall m}$.  Letting $M=AC$, the \textit{synthesis filter bank} associated with some unit norm vectors $\set{f_j}_{j\in\cJ}$ is $\set{\smash{\rmT^{Ai}f_j}}_{i=0, j\in\cJ}^{C-1}$, where $\rmT$ is the \textit{translation} operator $(\rmT f)(m):=f(m-1)$.  As one may verify that $\smash{FF^*\rmT^{Ai}=\rmT^{Ai}FF^*}$, Proposition~\ref{proposition.group invariance} guarantees that evolving the $f_j$'s according to Theorem~\ref{theorem.gradient of the frame potential} preserves this filter bank structure.  Letting $M=BD$, one can further consider the Gabor subclass of filter bank frames: the \textit{Gabor system} associated with some unit norm $f$ is $\smash{\set{\rmT^{Ai}\rmE^{Bj}f}_{i=0,\ j=0}^{C-1,D-1}}$, where $\rmE$ is the \textit{modulation} operator $(\rmE f)(m)=\rme^{\frac{2\pi\rmi m}M}f(m)$.  Though the operators $\rmE$ and $\rmT$ do not commute, we nevertheless have that $\rmE\rmT=\rme^{\frac{2\pi\rmi}M}\rmT\rmE$, a fact which suffices to guarantee that $FF^*\rmT^{Ai}\rmE^{Bj}=\rmT^{Ai}\rmE^{Bj}FF^*$, and so Proposition~\ref{proposition.group invariance} guarantees that the method of Theorem~\ref{theorem.gradient of the frame potential} preserves the Gabor structure.  In particular, one need only evolve $f$ itself, rather than the entirety of its modulates and translates.  That is, one need only compute
\begin{equation*}
FF^*f=\sum_{i=0}^{C-1}\sum_{j=0}^{D-1}\ip{f}{\rmT^{ai}\rmE^{bj}f}\rmT^{ai}\rmE^{bj}f
\end{equation*}
and consider $f(t)=\cos(\norm{g}t)f-\sin(\norm{g}t)\frac{g}{\norm{g}}$, where $g=FF^*f-\ip{FF^*f}{f}f$ and $t\in(0,\tfrac{1}{2N})$.  By iteratively applying this procedure, one produces Gabor frames of ever-increasing tightness.


\section{Sufficient conditions for linear convergence of gradient descent}

We now take a given unit norm sequence $F_0:=F=\set{f_n}_{n=1}^{N}$, and iteratively apply the main result of the previous section---Theorem~\ref{theorem.gradient of the frame potential}---to produce a sequence $\smash{\set{F_k}_{k=0}^{\infty}}$ of unit norm sequences of increasing tightness.  To be clear, fixing any $t\in(0,\tfrac{1}{2N})$, and given any unit norm sequence $F_k=\set{f_n^{(k)}}_{n=1}^{N}$, we first compute $G_k=\set{g_n^{(k)}}_{n=1}^{N}$:
\begin{equation}
\label{equation.definition of gradient descent 1}
g_n^{(k)}=P^{(k)}_n F_kF_kf_n^{(k)}=F_kF_kf_n^{(k)}-\ip{F_kF_kf_n^{(k)}}{f_n^{(k)}}f_n^{(k)},
\quad \forall n=1,\dotsc,N.
\end{equation}
We then define $F_k=\set{f_n^{(k+1)}}_{n=1}^{N}$ as follows: 
\begin{equation}
f\label{equation.definition of gradient descent 2}
_n^{(k+1)}:=\left\{\begin{array}{cl}\cos(\norm{g_n^{(k)}}t)f_n^{(k)}-\sin(\norm{g_n^{(k)}}t)\frac{g_n^{(k)}}{\norm{g_n^{(k)}}},& g_n^{(k)}\neq 0,\\f_n^{(k)},&g_n^{(k)} =0.\end{array}\right.
\end{equation}
While Theorem~\ref{theorem.gradient of the frame potential} guarantees that the values of $\norm{F_kF_k^*-\frac NM\rmI}_\HS$ are decreasing, it does not guarantee that this decrease is strict, nor that it decreases to zero in the limit, nor that the $F_k$'s themselves converge.  Indeed, gradient descent of the frame potential does not necessarily converge to a UNTF: despite the fact that every local minimizer of the frame potential is also a global minimizer, there do exist suboptimal \textit{critical frames} $F$ at which the gradient $G$ vanishes~\cite{BenedettoF:03}.  In this section, we provide conditions which suffice to avoid such nonoptimal critical frames, and moreover, guarantee that the iterative application of \eqref{equation.definition of gradient descent 1} and \eqref{equation.definition of gradient descent 2} produces a sequence of unit norm frames which indeed converges to a UNTF $F_\infty=\lim_k F_k$ that is close to $F=F_0$.  To do this, note that a unit norm sequence $F$ is critical with respect to the frame potential if and only if its gradient $G$ vanishes, which occurs precisely when each $f_n$ is an eigenvector of the frame operator $FF^*$.  As noted in~\cite{BenedettoF:03}, this occurs precisely when $F$ can be partitioned into a collection of subsequences, each of which is a unit norm tight frame for its span.  Here, the key is to recognize that in this setting, such orthogonality is actually one's enemy.  To be precise, we make the following definition:

\begin{defn}
A sequence $\set{f_n}_{n=1}^N\in\mathbb{S}_M^N$ is termed \emph{orthogonally partitionable (OP)} if there exists a nontrivial partition $\cI\sqcup\cJ=\set{1,\ldots,N}$ such that $\abs{\ip{f_i}{f_j}}=0$ for every $i\in\cI$, $j\in\cJ$.  More generally, it is $\varepsilon$-\emph{orthogonally partitionable ($\varepsilon$-OP)} if there exists a nontrivial partition $\cI\sqcup\cJ=\set{1,\ldots,N}$ such that $\abs{\ip{f_i}{f_j}}<\varepsilon$ for every $i\in\cI$, $j\in\cJ$.
\end{defn}

Thus, one way to ensure $G\neq0$ is to have that $F$ is not OP.  Indeed, as we show in the following result, if $F$ is not $\varepsilon$-OP, then the amount $F$'s frame potential decreases in one iteration of gradient descent, as given in Theorem~\ref{theorem.gradient of the frame potential}, is at least some fixed percentage of $F$'s distance from tightness.

\begin{thm}
\label{theorem.delta}
Let $\varepsilon\in(0,\tfrac1]$, and take $F\in\mathbb{S}_M^N$ satisfying $\norm{FF^*-\tfrac{N}{M}\rmI}_\HS\leq\tfrac{N}{2M}$.  Let $P_n$ denote the orthogonal projection from $\mathbb{H}_M$ onto the orthogonal complement of $f_n$.  If $F$ is not $\varepsilon$-orthogonally partitionable, then
\begin{equation}
\label{equation.alpha and beta}
\tfrac{\varepsilon^2}{4M^4}\bignorm{FF^*-\tfrac{N}{M}\rmI}_\HS^2
\leq \sum_{n=1}^N\norm{P_nFF^*f_n}^2
\leq 4N\bignorm{FF^*-\tfrac{N}{M}\rmI}_\HS^2.
\end{equation}
\end{thm}

\begin{proof}
Let $\set{\lambda_m}_{m=1}^M$ denote the eigenvalues of $FF^*$, arranged in increasing order, with corresponding orthonormal eigenbasis $\set{e_m}_{m=1}^M$.  Decomposing any $f_n$ in terms of this eigenbasis gives
\begin{equation*}
\gamma_n:=\ip{FF^*f_n}{f_n}=\biggip{FF^*\sum_{m=1}^M\ip{f_n}{e_m}e_m}{f_n}=\sum_{m=1}^M\lambda_m\abs{\ip{f_n}{e_m}}^2.
\end{equation*}
That is, each $\gamma_n$ is a convex combination of $FF^*$'s spectrum.  Since, as noted previously, $\tfrac{N}{M}$ is the average of the $\lambda_m$'s, we therefore have $\gamma_n,\tfrac{N}{M}\in[\lambda_{1},\lambda_{M}]$, and so for any $m$ and $n$,
\begin{equation}
\label{equation.projected length 1}
(\lambda_m-\gamma_n)^2\leq(\lambda_{M}-\lambda_{1})^2
\leq 4\max_{m'}(\lambda_{m'}-\tfrac{N}{M})^2
\leq 4\sum_{m'=1}^{M}(\lambda_{m'}-\tfrac{N}{M})^2
=4\bignorm{FF^*-\tfrac{N}{M}\rmI}_\HS^2.
\end{equation}
Also, by the definitions of $P_n$ and $\gamma_n$, we have $\sum_{n=1}^N\norm{P_nFF^*f_n}^2
=\sum_{n=1}^N\norm{(FF^*-\gamma_n\rmI)f_n}^2$.
Decomposing each $f_n$ in terms of the $e_m$'s therefore gives
\begin{equation}
\label{equation.projected length 2}
\sum_{n=1}^N\norm{P_nFF^*f_n}^2\!
=\sum_{n=1}^N\biggnorm{(FF^*-\gamma_n\rmI)\!\sum_{m=1}^M\ip{f_n}{e_m}e_m}^2\!
=\sum_{n=1}^N\biggnorm{\sum_{m=1}^M(\lambda_m-\gamma_n)\ip{f_n}{e_m}e_m}^2\!
=\sum_{n=1}^N\sum_{m=1}^M(\lambda_m-\gamma_n)^2\abs{\ip{f_n}{e_m}}^2.
\end{equation}
From here, we apply \eqref{equation.projected length 1} to get the right-hand inequality of~\eqref{equation.alpha and beta}:
\begin{align}
\nonumber
\label{equation.projected length 3}
\sum_{n=1}^N\norm{P_nFF^*f_n}^2
\leq 4\bignorm{FF^*-\tfrac{N}{M}\rmI}_\HS^2\sum_{n=1}^N\sum_{m=1}^M\abs{\ip{f_n}{e_m}}^2
=4N\bignorm{FF^*-\tfrac{N}{M}\rmI}_\HS^2.
\end{align}
Note that this inequality holds in general, that is, for any $F\in\bbS_M^N$.  We now seek the left-hand inequality of~\eqref{equation.alpha and beta}.  Since the largest gap between successive eigenvalues is no smaller than the average gap, there necessarily exists an $m_0$ that satisfies 
\begin{equation}
\label{equation.define m zero}
\lambda_{m_0+1}-\lambda_{m_0}
\geq\tfrac{1}{M-1}(\lambda_{M}-\lambda_{1})
\geq\tfrac{1}{M}(\lambda_{M}-\lambda_{1}).
\end{equation}
Define $\smash{\cI:=\set{n:\gamma_n<\tfrac{1}{2}(\lambda_{m_0}+\lambda_{m_0+1})}}$, $\smash{\cJ:=\set{1,\ldots,N}\setminus\cI}$.  This partitions the $\gamma_n$'s according to where they lie in relation to the midpoint $\tfrac{1}{2}(\lambda_{m_0}+\lambda_{m_0+1})$ of the largest gap between eigenvalues.  Therefore, the $\lambda_m$'s lying above this midpoint are at least half the gap away, namely at least $\tfrac{1}{2}(\lambda_{m_0+1}-\lambda_{m_0})\geq\tfrac{1}{2M}(\lambda_{M}-\lambda_{1})$ away, from the $\gamma_n$'s lying below the midpoint, and vice versa.  In fact, when $m\geq m_0+1$ and $n\in\cI$, or when $m\leq m_0$ and $n\in\cJ$, we have
\begin{equation}
\label{equation.partition gammas 1}
(\lambda_m-\gamma_n)^2
\geq\Bigbracket{\tfrac{1}{2M}(\lambda_{M}-\lambda_{1})}^2
\geq\tfrac{1}{4M^2}\max_m(\lambda_m-\tfrac{N}{M})^2
\geq\tfrac{1}{4M^3}\sum_m(\lambda_m-\tfrac{N}{M})^2
=\tfrac{1}{4M^3}\bignorm{FF^*-\tfrac{N}{M}\rmI}_\HS^2.
\end{equation}
That said, if $i\in\cI$ and $j\in\cJ$, then regardless of $m$, $\lambda_m$ is on one side of the midpoint $\tfrac{1}{2}(\lambda_{m_0}+\lambda_{m_0+1})$, and either $\gamma_i$ or $\gamma_j$ is on the other side, implying
\begin{equation}
\label{equation.partition gammas 1.5}
\max\biggset{(\lambda_m-\gamma_i)^2,(\lambda_m-\gamma_j)^2}
\geq\tfrac{1}{4M^3}\bignorm{FF^*-\tfrac{N}{M}\rmI}_\HS^2.
\end{equation}

Now suppose both $\cI$ and $\cJ$ are nonempty.  Since $F$ is not $\varepsilon$-OP, there exists $i\in\cI$ and $j\in\cJ$ such that $\varepsilon\leq\abs{\ip{f_i}{f_j}}$.  Decomposing over the eigenbasis, we therefore have
\begin{equation}
\label{equation.partition gammas 5}
\varepsilon^2
\leq\abs{\ip{f_i}{f_j}}^2
\leq\biggparen{\sum_{m=1}^M\abs{\ip{f_i}{e_m}}\abs{\ip{f_j}{e_m}}}^2\\
\leq M\sum_{m=1}^M\abs{\ip{f_i}{e_m}}^2\abs{\ip{f_j}{e_m}}^2
\leq M\sum_{m=1}^M\min\biggset{\abs{\ip{f_i}{e_m}}^2,\abs{\ip{f_j}{e_m}}^2},
\end{equation}
where the last inequality uses $\abs{\ip{f_n}{e_m}}\leq\norm{f_n}\norm{e_m}=1$.  Recalling \eqref{equation.projected length 2}, we isolate the $i$th and $j$th terms:
\begin{align*}
\sum_{n=1}^N\norm{P_nFF^*f_n}^2
&=\sum_{n=1}^N\sum_{m=1}^M(\lambda_m-\gamma_n)^2\abs{\ip{f_n}{e_m}}^2\\
&\geq\sum_{m=1}^M\biggparen{(\lambda_m-\gamma_i)^2\abs{\ip{f_i}{e_m}}^2+(\lambda_m-\gamma_j)^2\abs{\ip{f_j}{e_m}}^2}\\
&\geq\sum_{m=1}^M\max\biggset{(\lambda_m-\gamma_i)^2,(\lambda_m-\gamma_j)^2}\min\biggset{\abs{\ip{f_i}{e_m}}^2,\abs{\ip{f_j}{e_m}}^2}.
\end{align*}
From here, we apply \eqref{equation.partition gammas 1.5} and \eqref{equation.partition gammas 5} to get
\begin{equation*}
\sum_{n=1}^N\norm{P_nFF^*f_n}^2
\geq\tfrac{1}{4M^3}\bignorm{FF^*-\tfrac{N}{M}\rmI}_\HS^2\sum_{m=1}^M\min\biggset{\abs{\ip{f_i}{e_m}}^2,\abs{\ip{f_j}{e_m}}^2}
\geq\tfrac{\varepsilon^2}{4M^4}\bignorm{FF^*-\tfrac{N}{M}\rmI}_\HS^2.
\end{equation*}
Therefore, we indeed have the left-hand inequality of \eqref{equation.alpha and beta} in the case where both $\cI$ and $\cJ$ are nonempty.  We now turn to the case where either $\cI$ or $\cJ$ is empty.  We have
\begin{equation}
\label{equation.partition gammas 2}
\max_m(\lambda_m-\tfrac{N}{M})^2\leq\sum_{m=1}^M(\lambda_m-\tfrac{N}{M})^2
=\bignorm{FF^*-\tfrac{N}{M}}_\HS^2\leq\bigparen{\tfrac{N}{2M}}^2,
\end{equation}
where the last inequality follows from one of our assumptions.  Therefore, recalling $m_0$ from \eqref{equation.define m zero}, we have
\begin{equation}
\label{equation.partition gammas 3}
\sum_{n=1}^N\abs{\ip{f_n}{e_{m_0}}}^2
=\norm{F^*e_{m_0}}^2
=\ip{FF^*e_{m_0}}{e_{m_0}}
=\lambda_{m_0}
\geq\lambda_1
\geq\tfrac{N}{M}-\max_m\bigabs{\lambda_m-\tfrac{N}{M}}
\geq\tfrac{N}{2M},
\end{equation}
where the last inequality is by \eqref{equation.partition gammas 2}.  In particular, if $\cI$ is empty, we recall~\eqref{equation.projected length 2}, isolating its $m_0$th term:
\begin{equation}
\label{equation.partition gammas 3.5}
\sum_{n=1}^N\norm{P_nFF^*f_n}^2
=\sum_{n=1}^N\sum_{m=1}^M(\lambda_m-\gamma_n)^2\abs{\ip{f_n}{e_m}}^2
\geq\sum_{n=1}^N(\lambda_{m_0}-\gamma_n)^2\abs{\ip{f_n}{e_{m_0}}}^2.
\end{equation}
Since $\cI=\emptyset$, then $\cJ=\set{1,\dotsc,N}$, and thus~\eqref{equation.partition gammas 1} holds for $m=m_0$ and all $n$.  Coupled with~\eqref{equation.partition gammas 3} and~\eqref{equation.partition gammas 3.5}, this implies
\begin{equation*}
\sum_{n=1}^N\norm{P_nFF^*f_n}^2
\geq\tfrac{1}{4M^3}\bignorm{FF^*-\tfrac{N}{M}\rmI}_\HS^2\sum_{n=1}^N\abs{\ip{f_n}{e_{m_0}}}^2
\geq\tfrac{N}{8M^4}\bignorm{FF^*-\tfrac{N}{M}\rmI}_\HS^2
\geq\tfrac{\varepsilon^2}{4M^4}\bignorm{FF^*-\tfrac{N}{M}\rmI}_\HS^2,
\end{equation*}
where the last inequality uses $\varepsilon^2\leq1\leq\tfrac{N}{2}$.   This proves the left-hand inequality of~\eqref{equation.alpha and beta} in the case where $\cI$ is empty.  A similar argument---isolating the $(m_0+1)$st term in~\eqref{equation.projected length 2}---holds in the remaining case where $\cJ$ is empty.
\end{proof}

The previous result, along with Theorem~\ref{theorem.gradient of the frame potential}, guarantees a certain decrease in frame potential, provided the given frame $F$ is not $\varepsilon$-OP.  In the next result, we show that if, when performing the gradient descent steps~\eqref{equation.definition of gradient descent 1} and~\eqref{equation.definition of gradient descent 2}, one can ensure that each iteration $F_k$ is not $\varepsilon$-OP for some $\varepsilon>0$ independent of $k$, then gradient descent converges to a nearby UNTF at a linear rate.

\begin{thm}
\label{theorem.linear convergence}
Fix $\varepsilon\in(0,1]$ and $t\in(0,\tfrac{1}{2N})$, take $F_0=\set{f_n^{(0)}}_{n=1}^N\in\mathbb{S}_M^N$ satisfying $\norm{F_0F_0^*-\tfrac{N}{M}\rmI}_\HS\leq\tfrac{N}{2M}$, and iterate $F_{k+1}:=F_k(t)$ as in~\eqref{equation.definition of gradient descent 1} and~\eqref{equation.definition of gradient descent 2}.  If, for any fixed $K$, we have that $F_k$ is not $\varepsilon$-orthogonally partitionable for all $k=0,\ldots,K-1$, then the $K$th iteration $F_K$ satisfies
\begin{align}
\label{equation.linear convergence 1}
\norm{F_K-F_0}_\HS
&\leq\tfrac{4M^4N^\frac{1}{2}}{(1-2Nt)\varepsilon^2}\bignorm{F_0F_0^*-\tfrac{N}{M}\rmI}_\HS,\\
\label{equation.linear convergence 2}
\bignorm{F_KF_K^*-\tfrac{N}{M}\rmI}_\HS
&\leq\Bigparen{1-\tfrac{t(1-2Nt)\varepsilon^2}{M^4}}^{\frac K2}\bignorm{F_0F_0^*-\tfrac{N}{M}\rmI}_\HS.
\end{align}
Moreover, if $F_k$ is not $\varepsilon$-orthogonally partitionable for any $k$, then $F_\infty:=\lim_kF_k$ exists and is a unit norm tight frame within \eqref{equation.linear convergence 1} from $F_0$.
\end{thm}

\begin{proof}
Define $\gamma:=\tfrac{\varepsilon^2}{4M^4}$, and suppose $F_k$ is not $\varepsilon$-OP for $k=0,\ldots,K-1$.  Then combining~\eqref{equation.distance from tightness}, \eqref{equation.gradient descent 2} and the lower bound in~\eqref{equation.alpha and beta} gives that $F_{k+1}:=F_k(t)$ satisfies
\begin{align*}
\bignorm{F_{k+1}F_{k+1}^*-\tfrac{N}{M}\rmI}_\HS^2&=\FP(F_k(t))-\tfrac{N^2}{M}\\
&\leq\FP(F_k)-\tfrac{N^2}{M}-4t(1-2Nt)\sum_{n=1}^N\norm{P_n^{(k)}F_kF_k^*f_n^{(k)}}^2\\
&\leq\big[1-4t(1-2Nt)\gamma\big]\bignorm{F_kF_k^*-\tfrac{N}{M}\rmI}_\HS^2.
\end{align*}
From here, one may proceed inductively to find that
\begin{equation}
\label{equation.linear convergence 3}
\bignorm{F_kF_k^*-\tfrac{N}{M}\rmI}_\HS^2\leq\big[1-4t(1-2Nt)\gamma\big]^k\bignorm{F_0F_0^*-\tfrac{N}{M}\rmI}_\HS^2,
\end{equation}
which proves~\eqref{equation.linear convergence 2}, recalling $\gamma:=\tfrac{\varepsilon^2}{4M^4}$.  Next, let $\delta:=4N$.  To prove~\eqref{equation.linear convergence 1}, we use \eqref{equation.gradient descent 1}, the upper bound in~\eqref{equation.alpha and beta}, and \eqref{equation.linear convergence 3} to obtain
\begin{equation}
\label{equation.linear convergence 4}
\norm{F_{k+1}-F_k}_\HS^2
\leq t^2\sum_{n=1}^N\norm{P_n^{(k)}F_kF_k^*f_n^{(k)}}^2\\
\leq t^2\delta\bignorm{F_kF_k^*-\tfrac{N}{M}\rmI}_\HS^2\\
\leq t^2\delta\big[1-4t(1-2Nt)\gamma\big]^k\bignorm{F_0F_0^*-\tfrac{N}{M}\rmI}_\HS^2
\end{equation}
for all $k=0,\dotsc,K-1$.  In particular, for any $K'<K$, we can bound $\norm{F_{K}-F_{K'}}_\HS$ in terms of a geometric series; since $t\in(0,\tfrac{1}{2N})$ and $\gamma=\tfrac{\varepsilon^2}{4M^4}$ with $\varepsilon\in(0,1]$, this series is guaranteed to converge:
\begin{equation}
\label{equation.linear convergence 4.25}
\norm{F_K-F_{K'}}_\HS
\leq\sum_{k=K'}^{K-1}\norm{F_{k+1}-F_k}_\HS
\leq t\delta^\frac{1}{2}\biggparen{\sum_{k=K'}^\infty\bigbracket{1-4t(1-2Nt)\gamma}^\frac{k}{2}}\bignorm{F_0F_0^*-\tfrac{N}{M}\rmI}_\HS.
\end{equation}
In particular, letting $K'=0$ in~\eqref{equation.linear convergence 4.25} yields \eqref{equation.linear convergence 1}:
\begin{equation}
\label{equation.linear convergence 4.5}
\norm{F_K-F_0}_\HS
\leq\biggparen{\tfrac{t\delta^\frac{1}{2}}{1-[1-4t(1-2Nt)\gamma]^\frac{1}{2}}}\bignorm{F_0F_0^*-\tfrac{N}{M}\rmI}_\HS
\leq\tfrac{\delta^\frac{1}{2}}{2(1-2Nt)\gamma}\bignorm{F_0F_0^*-\tfrac{N}{M}\rmI}_\HS,
\end{equation}
where we have used the fact that $(1-x)^\frac{1}{2}\leq1-\tfrac{1}{2}x$.

Now suppose $F_k$ is never $\varepsilon$-OP for any $k$, and so \eqref{equation.linear convergence 4.25} holds for all $K'<K$.  In particular, as the series in~\eqref{equation.linear convergence 4.25} vanishes (independently of $K$) as $K'$ grows large, we have that $\set{F_k}_{k=0}^{\infty}$ is a Cauchy sequence.  As $\mathbb{S}_M^N$ is complete, $F_\infty:=\lim_k F_k$ exists.  Taking the limit of~\eqref{equation.linear convergence 3} yields $\smash{\norm{F_\infty F_\infty^*-\tfrac{N}{M}\rmI}_\HS=0}$, and so $F_\infty$ is a UNTF.  Meanwhile, taking the limit of~\eqref{equation.linear convergence 4.5} yields our final conclusion, namely that $F_\infty$ also satisfies \eqref{equation.linear convergence 1}:
\begin{equation*}
\norm{F_\infty-F_0}_\HS
\leq\tfrac{\delta^\frac{1}{2}}{2(1-2Nt)\gamma}\bignorm{F_0F_0^*-\tfrac{N}{M}\rmI}_\HS
=\tfrac{4M^4N^\frac{1}{2}}{(1-2Nt)\varepsilon^2}~\bignorm{F_0F_0^*-\tfrac{N}{M}\rmI}_\HS.\qedhere
\end{equation*}
\end{proof}


\section{Solutions to the Paulsen problem}

In the previous section, we applied gradient descent to $F_0\in\bbS_M^N$ to produce a sequence of iterates $\set{F_k}_{k=0}^\infty$.  We showed that if $F_0$ is sufficiently tight and if all resulting $F_k$'s are not $\varepsilon$-OP for some fixed $\varepsilon>0$, then this sequence converges to a UNTF at a linear rate.  In this section, we show that such an $\varepsilon$ always exists, provided $M$ and $N$ are relatively prime.  Meanwhile, in the non-relatively-prime case, we give an example that shows such $\varepsilon$'s are not guaranteed to exist.  In this case, our gradient descent algorithm's rate of convergence is threatened whenever our frame becomes nearly OP; to overcome this threat, we ``jump" from our current iterate to a nearby OP frame, and then continue gradient descent on the individual subframes over their respective subspaces.  In so doing, we are able to give solutions to the Paulsen problem \eqref{equation.our version of Paulsen} even in the non-relatively-prime case.

\subsection{Case I: $M$ and $N$ are relatively prime}

Theorem~\ref{theorem.linear convergence} guarantees that gradient descent converges to a UNTF at a linear rate, provided the iterations never become $\varepsilon$-OP for all arbitrarily small $\varepsilon$'s.  When $M$ and $N$ are relatively prime, this is not a problem:

\begin{thm}
\label{theorem.not orthogonal}
Take $F\in\mathbb{S}_M^N$ with $M$ and $N$ relatively prime.  If $\norm{FF^*-\tfrac{N}{M}\rmI}_\HS^2\leq\tfrac{2}{M^3}$, then $F$ is not $\bigparen{\tfrac{1}{M^8N^4}}$-orthogonally partitionable.
\end{thm}

\begin{proof}
We prove by contrapositive: take $\smash{F\in\mathbb{S}_M^N}$ with $M$ and $N$ relatively prime, and suppose $F$ is $\varepsilon$-OP with $\smash{\varepsilon:=\tfrac{1}{M^8N^4}}$;  we show that $\smash{\norm{FF^*-\tfrac{N}{M}\rmI}_\HS^2>\tfrac{2}{M^3}}$.  Since $F$ is $\varepsilon$-OP, there exists a nontrivial partition $\cI\sqcup\cJ=\set{1,\ldots,N}$ such that $\abs{\ip{f_i}{f_j}}<\varepsilon$ for every $i\in\cI$, $j\in\cJ$.  Define $F_\cI:=\set{f_i}_{i\in\cI}$ and $F_\cJ:=\set{f_j}_{j\in\cJ}$.  The frame operator $\smash{F_\cI F_\cI^*}$ has eigenvalues $\set{\lambda_{\cI,m}}_{m=1}^M$ and eigenvectors $\set{e_{\cI,m}}_{m=1}^M$, and similarly for $\smash{F_\cJ F_\cJ^*}$.  Without loss of generality, we arrange both sets of eigenvalues in decreasing order.  Take $\smash{\lambda':=\tfrac{1}{M^4N}}$, and define $M_\cI:=\#\set{m:\lambda_{\cI,m}\geq\lambda'}$, and similarly for $M_\cJ$.  We know $M_\cI\geq1$, since otherwise we have a contradiction:
\begin{equation*}
1\leq\abs{\cI}
=\Tr(F_\cI^*F_\cI)
=\Tr(F_\cI F_\cI^*)
=\sum_{m=1}^M\lambda_{\cI,m}
<M\lambda'
=\tfrac{1}{M^3N}
<1.
\end{equation*}
Similarly, $M_\cJ\geq1$.  Moreover, we claim $M_\cI+M_\cJ\leq M$.  Indeed, if not, then $\smash{\mathrm{Span}\set{e_{\cI,m}}_{m=1}^{M_\cI}\cap\mathrm{Span}\set{e_{\cJ,m}}_{m=1}^{M_\cJ}}$ has positive dimension, and so we may find a unit vector $u$ in this subspace.  Since $e_{\cI,m}$ is an eigenvector of $F_\cI F_\cI^*$ with eigenvalue $\lambda_{\cI,m}$, we have
\begin{equation*}
u
=\sum_{m=1}^{M_\cI}\ip{u}{e_{\cI,m}}e_{\cI,m}
=\sum_{m=1}^{M_\cI}\ip{u}{e_{\cI,m}}\tfrac{1}{\lambda_{\cI,m}}\sum_{i\in\cI}\ip{e_{\cI,m}}{f_i}f_i,
\end{equation*}
and we have a similar expression with $\cJ$.  Therefore, we apply the triangle inequality to get
\begin{align*}
1
&=\abs{\ip{u}{u}}^2
=\biggabs{\biggip{\sum_{m=1}^{M_\cI}\ip{u}{e_{\cI,m}}\tfrac{1}{\lambda_{\cI,m}}\sum_{i\in\cI}\ip{e_{\cI,m}}{f_i}f_i}{\sum_{m=1}^{M_\cJ}\ip{u}{e_{\cJ,m}}\tfrac{1}{\lambda_{\cJ,m}}\sum_{j\in\cJ}\ip{e_{\cJ,m}}{f_j}f_j}}\\
&\leq\sum_{i\in\cI}\sum_{m=1}^{M_\cI}\sum_{j\in\cJ}\sum_{m'=1}^{M_\cJ}\tfrac{\abs{\ip{f_i}{f_j}}}{\lambda_{\cI,m}\lambda_{\cJ,m'}}\abs{\ip{u}{e_{\cI,m}}}\abs{\ip{e_{\cI,m}}{f_i}}\abs{\ip{u}{e_{\cJ,m'}}}\abs{\ip{e_{\cJ,m'}}{f_j}}\\
&\leq\tfrac{\varepsilon}{(\lambda')^2}\sum_{i\in\cI}\biggparen{\sum_{m=1}^{M_\cI}\abs{\ip{u}{e_{\cI,m}}}\abs{\ip{e_{\cI,m}}{f_i}}}\sum_{j\in\cJ}\biggparen{\sum_{m=1}^{M_\cJ}\abs{\ip{u}{e_{\cJ,m}}}\abs{\ip{e_{\cJ,m}}{f_j}}},
\end{align*}
where the last inequality comes from $\abs{\ip{f_i}{f_j}}\leq\varepsilon$ and $\lambda_{\cI,m},\lambda_{\cJ,m'}\geq\lambda'$.  From here, we use $\tfrac{\varepsilon}{(\lambda')^2}=\tfrac{1}{N^2}$ and Holder's inequality to get
\begin{equation*}
1
\leq\tfrac{1}{N^2}\sum_{i\in\cI}\biggparen{\sum_{m=1}^{M_\cI}\abs{\ip{u}{e_{\cI,m}}}^2}^\frac{1}{2}\biggparen{\sum_{m=1}^{M_\cI}\abs{\ip{e_{\cI,m}}{f_i}}^2}^\frac{1}{2}\sum_{j\in\cJ}\biggparen{\sum_{m=1}^{M_\cJ}\abs{\ip{u}{e_{\cJ,m}}}^2}^\frac{1}{2}\biggparen{\sum_{m=1}^{M_\cJ}\abs{\ip{e_{\cJ,m}}{f_j}}^2}^\frac{1}{2}
\leq\tfrac{1}{N^2}\abs{\cI}\abs{\cJ}
\leq\tfrac{1}{4},
\end{equation*}
a contradiction.  As a partial summary, we know $M_\cI$ and $M_\cJ$ are nonzero and $M_\cI+M_\cJ\leq M$.  Now,
\begin{equation*}
\abs{\cI}
=\Tr(F_\cI^*F_\cI)
=\Tr(F_\cI F_\cI^*)
=\sum_{m=1}^M\lambda_{\cI,m}
=\sum_{m=1}^{M_\cI}\lambda_{\cI,m}+\!\!\!\sum_{m=M_\cI+1}^M\!\!\!\lambda_{\cI,m},
\end{equation*}
where $\sum_{m=M_\cI+1}^M\lambda_{\cI,m}<(M-M_\cI)\lambda'$.  Therefore, $\sum_{m=1}^{M_\cI}\lambda_{\cI,m}>\abs{\cI}-(M-M_\cI)\lambda'$, and so Jensen's inequality gives
\begin{equation}
\label{equation.not orthogonal 1}
\sum_{m=1}^{M_\cI}\lambda_{\cI,m}^2
\geq\tfrac{1}{M_\cI}\biggparen{\sum_{m=1}^{M_\cI}\lambda_{\cI,m}}^2
>\tfrac{1}{M_\cI}\Bigparen{\abs{\cI}-(M-M_\cI)\lambda'}^2
\geq\tfrac{\abs{\cI}^2}{M_\cI}-\tfrac{2\lambda'\abs{\cI}(M-M_\cI)}{M_\cI},
\end{equation}
and similarly for $\cJ$.  We now consider the frame potential of $F$:
\begin{equation*}
\FP(F)
=\Tr\bigbracket{(FF^*)^2}=\Tr\bigbracket{(F_\cI F_\cI^*+F_\cJ F_\cJ^*)^2}
=\Tr\bigbracket{(F_\cI F_\cI^*)^2}+\Tr\bigbracket{(F_\cJ F_\cJ^*)^2}+2\Tr\bigbracket{F_\cI F_\cI^*F_\cJ F_\cJ^*}.
\end{equation*}
Since $\Tr\bigbracket{F_\cI F_\cI^*F_\cJ F_\cJ^*}=\norm{F_\cI^*F_\cJ}_\HS^2\geq0$, we continue:
\begin{equation}
\label{equation.not orthogonal 2}
\FP(F)
\geq\sum_{m=1}^{M_\cI}\lambda_{\cI,m}^2+\sum_{m=1}^{M_\cJ}\lambda_{\cJ,m}^2\\
>\tfrac{\abs{\cI}^2}{M_\cI}+\tfrac{\abs{\cJ}^2}{M_\cJ}-2\lambda'\Bigbracket{\tfrac{\abs{\cI}(M-M_\cI)}{M_\cI}+\tfrac{\abs{\cJ}(M-M_\cJ)}{M_\cJ}},
\end{equation}
where the last inequality is by \eqref{equation.not orthogonal 1}.  Moreover, considering $M_\cI+M_\cJ\leq M$, we have
\begin{equation}
\label{equation.not orthogonal 3}
\tfrac{\abs{\cI}^2}{M_\cI}+\tfrac{\abs{\cJ}^2}{M_\cJ}
\geq\tfrac{\abs{\cI}^2}{M_\cI}+\tfrac{\paren{N-\abs{\cI}}^2}{M-M_\cI}
=\tfrac{N^2}{M}+\tfrac{\paren{\abs{\cI}M-M_\cI N}^2}{MM_\cI(M-M_\cI)}
\geq\tfrac{N^2}{M}+\tfrac{4}{M^3},
\end{equation}
where the last inequality uses the fact that $M$ and $N$ are relatively prime---that is, $\abs{\cI}M-M_\cI N$ is a nonzero integer---and $M_\cI(M-M_\cI)\leq\tfrac{M^2}{4}$.  Also, since $M_\cI,M_\cJ\geq1$, we have
\begin{equation}
\label{equation.not orthogonal 4}
\tfrac{\abs{\cI}(M-M_\cI)}{M_\cI}+\tfrac{\abs{\cJ}(M-M_\cJ)}{M_\cJ}
\leq(M-1)\bigparen{\abs{\cI}+\abs{\cJ}}
\leq MN.
\end{equation}
Therefore, combining \eqref{equation.not orthogonal 2}, \eqref{equation.not orthogonal 3} and \eqref{equation.not orthogonal 4} gives $\FP(F)>\tfrac{N^2}{M}+\tfrac{2}{M^3}$, meaning $\smash{\norm{FF^*-\tfrac{N}{M}\rmI}_\HS^2>\tfrac{2}{M^3}}$.
\end{proof}

Note that Theorem~\ref{theorem.not orthogonal} requires sufficient tightness to guarantee that $F$ is not $\bigparen{\tfrac{1}{M^8N^4}}$-othogonally partitionable.  Since gradient descent only decreases the frame potential, Theorem~\ref{theorem.not orthogonal} will apply to every subsequent iteration.  Therefore, by Theorem~\ref{theorem.linear convergence}, gradient descent converges to a UNTF in the relatively prime case:

\begin{cor}
\label{corollary.relatively prime}
Suppose $M$ and $N$ are relatively prime.  Pick $t\in(0,\tfrac{1}{2N})$, take $F_0\in\mathbb{S}_M^N$ satisfying $\norm{F_0F_0^*-\tfrac{N}{M}\rmI}_\HS^2\leq\tfrac{2}{M^3}$, and iterate $F_{k+1}:=F_k(t)$ as in~\eqref{equation.definition of gradient descent 1} and~\eqref{equation.definition of gradient descent 2}.  Then, $F_\infty:=\lim_kF_k$ exists and is a unit norm tight frame satisfying
\begin{equation*}
\norm{F_\infty-F_0}_\HS\leq\tfrac{4M^{20}N^{8.5}}{1-2Nt}\bignorm{F_0F_0^*-\tfrac{N}{M}\rmI}_\HS.
\end{equation*}
\end{cor}

This solves the Paulsen problem \eqref{equation.our version of Paulsen} in the case where $M$ and $N$ are relatively prime.  To be explicit, taking $t=\tfrac{1}{4N}$, we have $\smash{\delta=2^\frac{1}{2}M^{-\frac{3}{2}}}$, $\smash{C=8M^{20}N^{8.5}}$, and $\alpha=1$.  These constants are roughly comparable to those previously given in~\cite{BodmannC:10}, which were obtained using independent methods.  As noted earlier, $\alpha=1$ is the best one can hope for in any case.  In the next subsection, we give an example that shows that these techniques fall apart in the case where $M$ and $N$ share a common divisor, and moreover, that in such cases, we must set our sights lower with respect to $\alpha$.

\subsection{Case II: $M$ and $N$ are not relatively prime}

We continue our solution to the Paulsen problem in the remaining case where $M$ and $N$ are not relatively prime.  Let's begin this case with an example in two dimensions:

\begin{exmp}
Take some real $F\in\mathbb{S}_2^N$, that is, $F=\set{(\cos\theta_n,\sin\theta_n)}_{n=1}^N$ for some collection of $\theta_n$'s.  In this case, it is known~\cite{GoyalKK:01} that $F$ is tight precisely when the sum of $\set{(\cos2\theta_n,\sin2\theta_n)}_{n=1}^N$ vanishes.  In fact, one can show that
\begin{equation*}
\FP(F)-\tfrac{N^2}{2}
=\biggparen{\sum_{n=1}^N\cos^2\theta_n}^2+2~\biggparen{\sum_{n=1}^N\cos\theta_n\sin\theta_n}^2+\biggparen{\sum_{n=1}^N\sin^2\theta_n}^2-\tfrac{N^2}{2}
=\tfrac{1}{2}\biggbracket{\biggparen{\sum_{n=1}^N\cos2\theta_n}^2+\biggparen{\sum_{n=1}^N\sin2\theta_n}^2},
\end{equation*}
and so $\smash{\norm{FF^*-\tfrac{N}{2}\rmI}_\HS=\frac{1}{\sqrt{2}}\norm{\sum_{n=1}^N\paren{\cos2\theta,\sin2\theta}}}$.  That is, given any unit vectors in $\mathbb{R}^2$, double their polar angles, and add the resulting vectors, base-to-tip; for this chain of vectors, the distance between its head and tail is proportional to the original vectors' distance from tightness.  In particular, our physical intuition tells us that if a collection of unit vectors is close to being tight, then their double-angle counterparts must only be slightly perturbed in order to close their chain, meaning the original vectors are indeed close to a UNTF.  But how close?  To begin to answer this question, consider the following example:
\begin{equation}
\label{equation.example 1}
F(\theta):=\biggbracket{\begin{array}{rrrr}\cos\theta&\cos\theta&0&0\\\sin\theta&-\sin\theta&1&1\end{array}},
\qquad\qquad
\tilde{F}(\theta):=\biggbracket{\begin{array}{rrrr}\cos\tfrac\theta2&\cos\tfrac\theta2&-\sin\tfrac\theta2&\sin\tfrac\theta2\\\sin\tfrac\theta2&-\sin\tfrac\theta2&\cos\tfrac\theta2&\cos\tfrac\theta2\end{array}}.
\end{equation}
One can show that $\smash{\norm{F(\theta)F^*(\theta)-\tfrac{N}{2}\rmI}_\HS^2=8\sin^4\theta}$, while $\smash{\sum_{n=1}^N\norm{P_n(\theta)F(\theta)F^*(\theta)f_n(\theta)}^2=32\sin^6\theta\cos^2\theta}$.  That said, unlike in~\eqref{equation.alpha and beta}, there is no factor $A$ independent of $\theta$ such that $A\norm{F(\theta)F^*(\theta)-\tfrac{N}{2}\rmI}_\HS^2\leq\sum_{n=1}^N\norm{P_n(\theta)F(\theta)F^*(\theta)f_n(\theta)}^2$ for all $\theta$.  Therefore, at the very least, our analysis of the gradient descent algorithm, given in the previous section, must be refined in order to guarantee convergence.

Nevertheless, in this example, we can show that gradient descent does, in fact, converge to a UNTF, albeit at a sublinear rate.  Here, $g_1(\theta)=4\cos\theta\sin^3\theta(-\sin\theta,\cos\theta)$, $g_2(\theta)=-4\cos\theta\sin^3\theta(\sin\theta,\cos\theta)$, and $g_3(\theta)=g_4(\theta)=0$.  Recalling Proposition~\ref{proposition.Taylor's theorem}, one can show that $F(\theta;t)=F(\theta-4t\cos\theta\sin^3\theta)$.  That is, each iteration transforms an arrangement of angle $\theta$ into a new arrangement with angle $\theta-4t\cos\theta\sin^3\theta$; repeated iterations indeed converge to $\theta=0$, albeit very slowly.  In this way, gradient descent converges to $\set{e_1,e_1,e_2,e_2}$, that is, two copies of the standard basis, which is indeed a UNTF.  Note that since the limiting frame is OP, we know that for each $\varepsilon>0$, the $F_k$'s eventually become $\varepsilon$-OP---this is why the linear rate of convergence guaranteed by Theorem~\ref{theorem.linear convergence} does not hold here.

This same example can be used to give a baseline on answers to the Paulsen problem in the non-relatively-prime case.  Indeed, noting that every real UNTF in $\mathbb{S}_2^4$ is the union of two orthonormal bases, we can show that for each $\theta\in[0,\tfrac{\pi}{8}]$, $\tilde{F}(\theta)$ is the closest UNTF to $F(\theta)$.  But, $\smash{\norm{\tilde{F}(\theta)-F(\theta)}_\HS=4\sin\tfrac\theta4}$, which is on the order of the square-root of $\smash{\norm{F(\theta)F^*(\theta)-\tfrac{N}{2}\rmI}_\HS^\frac{1}{2}}$ as $\theta$ grows small.  As such, \eqref{equation.example 1} is a counterexample to the sometimes-voiced belief that distance from a UNTF is at worst a linear function of distance from tightness.  In other words, recalling \eqref{equation.our version of Paulsen}, $\alpha=1$ is not possible for every $M$ and $N$; even when $M=2$ and $N=4$, the best possible $\alpha$ is $\frac12$.  This leads to three important questions: 1) For a given $M$ and $N$, is the version of the Paulsen problem given in~\eqref{equation.our version of Paulsen} even solvable? 2) If so, what is the best possible $\alpha$ for a given $M$ and $N$? 3) Is there a single $\alpha$ that works for all $M$ and $N$, or does performance truly depend on the number of common factors between $M$ and $N$?  Below, we outline an argument that answers the first question in the affirmative; the second and third questions remain open.
\end{exmp}

As the preceeding example illustrated, gradient descent is not guaranteed to converge in the non-relatively-prime case, since there is no $\varepsilon$ for which iterations never become $\varepsilon$-OP.  To resolve this issue, we introduce the concept of ``jumping'' to a nearby OP unit norm frame:

\begin{thm}
\label{theorem.close orthogonal}
Let $\varepsilon\in(0,\tfrac{1}{2M}]$. Then, for every $\varepsilon$-orthogonally partitionable $F\in\mathbb{S}_M^N$, there exists an orthogonally partitionable $\tilde{F}\in\mathbb{S}_M^N$ such that $\smash{\norm{\tilde{F}-F}_\HS\leq(2N)^{\frac{1}{2}}(M\varepsilon)^{\frac{1}{3}}}$.
\end{thm}

\begin{proof}
We first claim that for every unit vector $f\in\bbH_M$ and every nonzero projection operator $P$ on $\bbH_M$, there exists a unit vector $g\in P(\bbH_M)$ such that $\norm{f-g}^2\leq2\norm{(I-P)f}^2$.  If $Pf=0$, we may take $g$ to be any unit vector in $P(\bbH_M)$,  since that would mean $\norm{f-g}^2=2=2\norm{f}^2=2\norm{(I-P)f}^2$.Otherwise, we take $g=\tfrac{Pf}{\norm{Pf}}$, since
\begin{equation*}
\bignorm{f-\tfrac{Pf}{\norm{Pf}}}^2
=\bignorm{Pf+(I-P)f-\tfrac{Pf}{\norm{Pf}}}^2
=\bignorm{\bigparen{1-\tfrac{1}{\norm{Pf}}}Pf+(I-P)f}^2,
\end{equation*}
and so the Pythagorean theorem gives
\begin{equation}
\label{equation.pure orthogonal 1}
\bignorm{f-\tfrac{Pf}{\norm{Pf}}}^2
=\bigparen{1-\tfrac{1}{\norm{Pf}}}^2\norm{Pf}^2+\norm{(I-P)f}^2
=2\bigparen{1-\norm{Pf}}
\leq2(1-\norm{Pf}^2)
=2\norm{(I-P)f}^2.
\end{equation}
For simplicity, we take $g:=\tfrac{Pf}{\norm{Pf}}$, understanding what this means when $Pf=0$.

Since $F$ is $\varepsilon$-OP, we have $\cI\sqcup\cJ=\set{1,\ldots,N}$ such that $\abs{\ip{f_i}{f_j}}<\varepsilon$ whenever $i\in\cI$ and $j\in\cJ$.  Without loss of generality, we take $\abs{\cI}\geq\abs{\cJ}$.  Defining $F_\cI:=\set{f_i}_{i\in\cI}$, the frame operator $F_\cI F_\cI^*$ has eigenvalues $\set{\lambda_{\cI,m}}_{m=1}^M$, arranged in decreasing order, and eigenvectors $\set{e_{\cI,m}}_{m=1}^M$.  Take $\smash{\lambda':=\tfrac{2N}{3}\bigparen{\tfrac{\varepsilon^2}{M}}^\frac{1}{3}}$, and define $M_\cI:=\#\set{m:\lambda_{\cI,m}\geq\lambda'}$.  We know $M_\cI\geq1$, since otherwise
\begin{equation*}
\tfrac{N}{2}\leq\abs{\cI}
=\Tr(F_\cI^*F_\cI)
=\Tr(F_\cI F_\cI^*)
=\sum_{m=1}^M\lambda_{\cI,m}
<M\lambda'
=\tfrac{2N}{3}\bigparen{M\varepsilon}^\frac{2}{3}
\leq\tfrac{2^\frac{1}{3}N}{3}
<\tfrac{N}{2}.
\end{equation*}
Therefore, $P:=\sum_{m=1}^{M_\cI}e_{\cI,m}e_{\cI,m}^*$ is a nonzero projection operator on $\bbH_M$.  Moreover,
\begin{equation}
\label{equation.pure orthogonal 2}
\sum_{i\in\cI}\norm{(I-P)f_i}^2
=\sum_{i\in\cI}\sum_{m=M_\cI+1}^M\!\!\!\abs{\ip{f_i}{e_{\cI,m}}}^2
=\!\!\!\sum_{m=M_\cI+1}^M\!\!\!\norm{F_\cI^*e_{\cI,m}}^2
=\!\!\!\sum_{m=M_\cI+1}^M\!\!\!\ip{F_\cI F_\cI^*e_{\cI,m}}{e_{\cI,m}}
=\!\!\!\sum_{m=M_\cI+1}^M\!\!\!\lambda_{\cI,m}
<M\lambda'.
\end{equation}
Also, the fact that $e_{\cI,m}$ is an eigenvector of $F_\cI F_\cI^*$ with eigenvalue $\lambda_{\cI,m}$ gives
\begin{equation*}
\sum_{j\in\cJ}\norm{Pf_j}^2
=\sum_{j\in\cJ}\sum_{m=1}^{M_\cI}\abs{\ip{f_j}{e_{\cI,m}}}^2
=\sum_{j\in\cJ}\sum_{m=1}^{M_\cI}\biggabs{\biggip{f_j}{\tfrac{1}{\lambda_{\cI,m}}\sum_{i\in\cI}\ip{e_{\cI,m}}{f_i}f_i}}^2
\leq\sum_{j\in\cJ}\sum_{m=1}^{M_\cI}\tfrac{1}{\lambda_{\cI,m}^2}\biggparen{\sum_{i\in\cI}\abs{\ip{e_{\cI,m}}{f_i}}\abs{\ip{f_i}{f_j}}}^2.
\end{equation*}
Continuing, we use $\abs{\ip{f_i}{f_j}}\leq\varepsilon$ and $\lambda_{\cI,m}\geq\lambda'$:
\begin{equation}
\label{equation.pure orthogonal 3}
\sum_{j\in\cJ}\norm{Pf_j}^2
\leq\tfrac{\varepsilon^2}{(\lambda')^2}\sum_{j\in\cJ}\sum_{m=1}^{M_\cI}\biggparen{\sum_{i\in\cI}\abs{\ip{e_{\cI,m}}{f_i}}}^2
\leq\tfrac{\varepsilon^2}{(\lambda')^2}\abs{\cI}\sum_{i\in\cI}\sum_{j\in\cJ}\sum_{m=1}^{M_\cI}\abs{\ip{e_{\cI,m}}{f_i}}^2
\leq\tfrac{\varepsilon^2}{(\lambda')^2}\abs{\cI}^2\abs{\cJ}
\leq\tfrac{4N^3\varepsilon^2}{27(\lambda')^2},
\end{equation}
where the last inequality comes from $\abs{\cI}^2(N-\abs{\cI})\leq\tfrac{4N^3}{27}$.  Define $\tilde{F}=\set{\tilde{f}_n}_{n=1}^N$ by $\smash{\tilde{f}_n=\frac{Pf_n}{\norm{Pf_n}}}$ when $n\in\cI$, and $\smash{\tilde{f}_n=\frac{(I-P)f_n}{\norm{(I-P)f_n}}}$ when $n\in\cJ$.  Then, combining \eqref{equation.pure orthogonal 1} with \eqref{equation.pure orthogonal 2} and \eqref{equation.pure orthogonal 3} gives the result:
\begin{equation*}
\norm{\tilde{F}-F}_\HS^2
=\sum_{i\in\cI}\bignorm{f_i-\tfrac{Pf_i}{\norm{Pf_i}}}^2+\!\sum_{j\in\cJ}\bignorm{f_j-\tfrac{(I-P)f_j}{\norm{(I-P)f_j}}}^2
\leq\sum_{i\in\cI}2\norm{(I-P)f_i}^2+\!\sum_{j\in\cJ}2\norm{Pf_j}^2
<2M\lambda'+\tfrac{8N^3\varepsilon^2}{27(\lambda')^2}
=2N(M\varepsilon)^\frac{2}{3}.
\qedhere
\end{equation*}
\end{proof}

The previous result tells us how far we must jump in order to transform an $\varepsilon$-OP frame into one that is exactly OP.  This opens the door for the following procedure for producing UNTFs in the non-relatively-prime case: given a collection of unit norm vectors and fixing any $\varepsilon\in(0,1]$, perform gradient descent until one's vectors become $\varepsilon$-OP, at which jump to a OP frame, and then repeat this procedure on each of the two subframes.  In the following result, we use Theorems~\ref{theorem.linear convergence} and~\ref{theorem.close orthogonal} to bound how far this procedure will take us from our original frame.

\begin{thm}
\label{theorem.not relatively prime}
Suppose $M$ and $N$ are not relatively prime.
Take $F\in\mathbb{S}_M^N$ such that $\norm{FF^*-\tfrac{N}{M}\rmI}_\HS\leq(2^{21}M^{27}N^{14})^{-1}$.
Then there exists $\tilde{F}\in\mathbb{S}_M^N$, which is either a unit norm tight frame or is orthogonally partitionable, with equal redundancies in each of the two partitioned subspaces, such that
\begin{equation}
\label{equation.untf or op}
\norm{\tilde{F}-F}_\HS\leq3M^\frac{6}{7}N^\frac{1}{2}\bignorm{FF^*-\tfrac{N}{M}\rmI}_\HS^\frac{1}{7}.
\end{equation}
\end{thm}

\begin{proof}
Take $t:=\tfrac{1}{4N}$ and $\smash{\varepsilon:=2^\frac{3}{2}3^\frac{3}{7}M^\frac{11}{7}\norm{FF^*-\tfrac{M}{N}\rmI}_\HS^\frac{3}{7}}$.  According to Theorem~\ref{theorem.linear convergence}, gradient descent will converge to a UNTF, provided iterations never become $\varepsilon$-OP.  In this way, we either converge to a UNTF $\tilde{F}$, or produce an $\varepsilon$-OP frame within $(2N)^\frac{1}{2}(M\varepsilon)^\frac{1}{3}$ of an OP frame $\tilde{F}$, by Theorem~\ref{theorem.close orthogonal}.Either way, Theorems~\ref{theorem.linear convergence} and~\ref{theorem.close orthogonal} give
\begin{equation*}
\norm{\tilde{F}-F}_\HS
\leq\tfrac{8M^4N^\frac{1}{2}}{\varepsilon^2}\bignorm{FF^*-\tfrac{N}{M}\rmI}_\HS+(2N)^\frac{1}{2}(M\varepsilon)^\frac{1}{3}
=3^{-\frac{6}{7}}7M^\frac{6}{7}N^\frac{1}{2}\bignorm{FF^*-\tfrac{N}{M}\rmI}_\HS^\frac{1}{7},
\end{equation*}
which proves \eqref{equation.untf or op}.  Now suppose $\tilde{F}$ is OP.  Since
\begin{align*}
\bigabs{\FP(\tilde{F})-\FP(F)}
&=\Tr\bigbracket{(\tilde{F}\tilde{F}^*-FF^*)(\tilde{F}\tilde{F}^*+FF^*)}\\
&\leq\norm{\tilde{F}\tilde{F}^*-FF^*}_\HS\norm{\tilde{F}\tilde{F}^*+FF^*}_\HS\\
&\leq\norm{\tilde{F}-F}_\HS\Bigparen{\norm{\tilde{F}}_\HS+\norm{F}_\HS}\Bigparen{\norm{\tilde{F}}_\HS^2+\norm{F}_\HS^2},
\end{align*}
we use $\norm{F}_\HS^2=\norm{\tilde{F}}_\HS^2=N$ to get $\abs{\FP(\tilde{F})-\FP(F)}\leq4N^\frac{3}{2}\norm{\tilde{F}-F}_\HS$.  Therefore,
\begin{equation*}
\FP(\tilde{F})
\leq\FP(F)+\bigabs{\FP(\tilde{F})-\FP(F)}
=\tfrac{N^2}{M}+\bignorm{FF^*-\tfrac{N}{M}\rmI}_\HS^2+\bigabs{\FP(\tilde{F})-\FP(F)}
\leq\tfrac{N^2}{M}+\bignorm{FF^*-\tfrac{N}{M}\rmI}_\HS^2+4N^\frac{3}{2}\norm{\tilde{F}-F}_\HS.
\end{equation*}
Continuing, we apply \eqref{equation.untf or op} and use the fact that $\norm{FF^*-\tfrac{N}{M}\rmI}_\HS^2\leq4N^\frac{3}{2}\bigparen{3M^\frac{6}{7}N^\frac{1}{2}\norm{FF^*-\tfrac{N}{M}\rmI}_\HS^\frac{1}{7}}$:
\begin{equation}
\label{equation.below threshold}
\FP(\tilde{F})
\leq\tfrac{N^2}{M}+\bignorm{FF^*-\tfrac{N}{M}\rmI}_\HS^2+4N^\frac{3}{2}\biggparen{3M^\frac{6}{7}N^\frac{1}{2}\bignorm{FF^*-\tfrac{N}{M}\rmI}_\HS^\frac{1}{7}}
\leq\tfrac{N^2}{M}+\tfrac{24M^\frac{6}{7}N^2}{(2^{21}M^{27}N^{14})^\frac{1}{7}}
=\tfrac{N^2}{M}+\tfrac{3}{M^3}.
\end{equation}
Since $\tilde{F}$ is OP, there exists an orthogonal partition $\cI\sqcup\cJ=\set{1,\ldots,N}$.  Take $M_\cI$ to be the dimension of the span of $\set{f_n}_{n\in\cI}$.  Then,
\begin{align*}
\FP(\tilde{F})
=\FP(\tilde{F}_\cI)+\FP(\tilde{F}_\cJ)
\geq\tfrac{\abs{\cI}^2}{M_\cI}+\tfrac{\paren{N-\abs{\cI}}^2}{M-M_\cI}
=\tfrac{N^2}{M}+\tfrac{\paren{\abs{\cI}M-M_\cI N}^2}{MM_\cI(M-M_\cI)}.
\end{align*}
In particular, if $\abs{\cI}M-M_\cI N\neq0$, then $\bigparen{\abs{\cI}M-M_\cI N}^2\geq1$, and since $M_\cI(M-M_\cI)\leq\tfrac{M}{4}$, we would have $\smash{\FP(\tilde{F})\geq\tfrac{N^2}{M}+\frac{4}{M^3}}$.  Considering \eqref{equation.below threshold}, we may conclude that $\abs{\cI}M-M_\cI N=0$, and so $\tfrac{N}{M}=\tfrac{\abs{\cI}}{M_\cI}=\frac{N-\abs{\cI}}{M-M_\cI}$.
\end{proof}

Repeated applications of Theorem~\ref{theorem.not relatively prime} will provide solutions, albeit inelegant ones, to the Paulsen problem given in~\eqref{equation.our version of Paulsen}.  To elaborate, Theorem~\ref{theorem.not relatively prime} states that if a unit norm frame $F$ is sufficiently tight, then there exists a unit norm $\tilde{F}$ such that $\smash{\norm{\tilde{F}-F}_\HS=\mathrm{O}(\norm{FF^*-\frac NM\rmI}^{\frac17})}$ which is either a UNTF or is OP into components of equal redundancy.  Since we are done if $\tilde{F}$ happens to be a UNTF, let's focus on the case where $\tilde{F}$ is OP, that is, when $\tilde{F}=\tilde{F}_{\cI}\oplus\tilde{F}_{\cJ}$, where $\tilde{F}_\cI=\set{\tilde{f}_i}_{i\in\cI}$ and $\tilde{F}_\cJ=\set{\tilde{f}_j}_{j\in\cJ}$ are frames for some $M_\cI$- and $M_\cJ$-dimensional subspaces of $\bbH_M$, respectively, and $\smash{\frac{\abs{\cI}}{M_{\cI}}=\frac{\abs{\cJ}}{M_{\cJ}}=\frac NM}$.  We then apply Theorem~\ref{theorem.not relatively prime} to $\tilde{F}_{\cI}$ and $\tilde{F}_{\cJ}$: if each is close to a UNTF, these can be directly summed to form a UNTF which is close to $\tilde{F}$ and in turn, to $F$; if either is OP, we must continue this process in lower-dimensional subspaces.  At most $M$ such nested applications of Theorem~\ref{theorem.not relatively prime} are necessary, since each reduces the dimension of the space in consideration by at least $1$.  The main issue is that each application of Theorem~\ref{theorem.not relatively prime} comes at a terrible cost: ``jumping" from an $\varepsilon$-OP sequence to an OP sequence can increase one's frame potential by a constant multiple of the jump distance.  In particular, with each application of Theorem~\ref{theorem.not relatively prime}, one's distance from tightness may be effectively raised to a $\frac17$ power; when one's distance is very small, this exponentiation results in a dramatic increase in distance.  When applied $M$ times in succession, one would therefore expect a net exponent of $\smash{\frac1{7^M}}$.  That is, we expect that there exists an extremely small $\delta>0$ and an extremely large $C$ for which~\eqref{equation.our version of Paulsen} will hold for $\smash{\alpha=\frac1{7^M}}$.  It is unknown whether such an $M$-dependent $\alpha$ is inherent to this problem, or simply a consequence of a weak argument on our part.

We emphasize that such issues, while of great mathematical interest, should cause little worry in real-world applications.  Indeed, the ``perform gradient descent and jump when approaching OP" method that we employed in the proof of Theorem~\ref{theorem.not relatively prime} produces UNTFs which, for all practical purposes, are close to their originals.  Nevertheless, the issue stands: this distance may not be a nice function of the tightness itself.  Indeed, this is the heart of the part of the Paulsen problem that remains open: ``Given a unit norm frame which is extremely close to being tight, and is also extremely close to being OP, how far away, as a function of tightness, is the nearest UNTF?''  This problem reveals our current lack of understanding of the geometry of the set of all UNTFs on very small neighborhoods of OP UNTFs, and is more than worthy of additional study.

\section*{Acknowledgments}
Casazza was supported by NSF DMS 0704216 and 1008183.  Fickus was supported by AFOSR F1ATA09125G003.  The views expressed in this article are those of the authors and do not reflect the official policy or position of the United States Air Force, Department of Defense, or the U.S. Government.

\end{document}